\documentclass[12pt, a4paper,reqno]{amsart}

\usepackage[euler-digits]{eulervm}
\usepackage{graphicx}
\usepackage{amsfonts, amsthm, amssymb, amsmath, stmaryrd}
\usepackage[headings]{fullpage}

\usepackage{ytableau,enumitem}
\usepackage{color}
\usepackage{mathrsfs}
\usepackage{amssymb}
\usepackage{ulem}
\usepackage{hyperref}
\usepackage[all,cmtip]{xy}
\usepackage{comment}

\newcommand{\ncom}{\newcommand}

\ncom{\dho}{\partial}
\ncom{\rar}{\rightarrow}
\ncom{\imply}{\Rightarrow}
\ncom{\lrar}{\longrightarrow}
\ncom{\into}{\hookrightarrow}
\ncom{\onto}{\twoheadrightarrow}
\ncom{\ov}{\overline}
\ncom{\m}{\mbox}
\ncom{\sta}{\stackrel}
\ncom{\invlim}{\varprojlim}
\ncom{\xhat}{\widehat}

\ncom{\vspc}{\vspace{3mm}}
\ncom{\End}{{\cE}nd}
\ncom{\tensor}{\otimes}

\ncom{\al}{\alpha}
\ncom{\cHom}{{\mathcal Hom}}

\ncom{\A}{{\mathbb A}}
\ncom{\comx}{{\mathbb C}}
\ncom{\E}{{\mathbb E}}
\ncom{\F}{{\mathbb F}}
\ncom{\G}{{\mathbb G}}
\ncom{\K}{{\mathbb K}}
\ncom{\Le}{{\mathbb L}}
\ncom{\N}{{\mathbb N}}

\ncom{\p}{{\mathbb P}}
\ncom{\Q}{{\mathbb Q}}
\ncom{\R}{{\mathbb R}}
\ncom{\Z}{{\mathbb Z}}

\ncom{\f}{\dfrac}

\ncom{\wtil}{\widetilde}

\ncom{\ci}{{\mathpzc i}}

\ncom{\cA}{{\mathcal A}}
\ncom{\cB}{{\mathcal B}}
\ncom{\cC}{{\mathcal C}}
\ncom{\cD}{{\mathcal D}}
\ncom{\cE}{{\mathcal E}}
\ncom{\cF}{{\mathcal F}}
\ncom{\cG}{{\mathcal G}}
\ncom{\cH}{{\mathcal H}}
\ncom{\cI}{{\mathcal I}}
\ncom{\cJ}{{\mathcal J}}
\ncom{\cK}{{\mathcal K}}
\ncom{\cL}{{\mathcal L}}
\ncom{\cM}{{\mathcal M}}
\ncom{\cN}{{\mathcal N}}
\ncom{\cO}{{\mathcal O}}
\ncom{\cP}{{\mathcal P}}
\ncom{\cQ}{{\mathcal Q}}
\ncom{\cR}{{\mathcal R}}
\ncom{\cS}{{\mathcal S}}
\ncom{\cT}{{\mathcal T}}
\ncom{\cU}{{\mathcal U}}
\ncom{\cV}{{\mathcal V}}
\ncom{\cW}{{\mathcal W}}
\ncom{\cX}{{\mathcal X}}
\ncom{\cY}{{\mathcal Y}}
\ncom{\cZ}{{\mathcal Z}}

\ncom{\cSU}{{\mathcal S \mathcal U}}
\ncom{\eop}{{\hfill $\Box$}}
\ncom{\isom}{\cong}

\theoremstyle{plain}
\newtheorem{theorem}{Theorem}[section]
\newtheorem{lemma}[theorem]{Lemma}

\newtheorem{proposition}[theorem]{Proposition}
\theoremstyle{definition}

\theoremstyle{remark}
\newtheorem{remark}[theorem]{Remark}

\long\def\comment#1{}

\newtheorem{exmp}[theorem]{Example}

\numberwithin{equation}{section}

\begin{document}

	\title{Quasilifting of hulls and depth of tensor product of modules}
	
	\author{Sutapa Dey}
	\address{Department of Mathematics, Indian Institute of Technology -- Hyderabad, 502285, India.}
	\email{ma20resch11002@iith.ac.in}

	\author{Amit Tripathi}
	\address{Department of Mathematics, Indian Institute of Technology -- Hyderabad, 502285, India.}
	\email{amittr@gmail.com}

	\begin{abstract} We investigate the depth of the tensor product of finitely generated modules over local rings. One of the main ingredients of our approach is a lifting construction introduced by Huneke, Jorgensen and Wiegand. We recover a result of Celikbas, Sadeghi, and Takahashi for local complete intersection rings. Additionally, we provide a negative answer to a question they asked and establish a corresponding lower bound. We derive a result on the depth of the tensor product of certain modules over local complete $\mathcal{TE}$ rings. Some general conditions on the existence of hulls and approximations are also studied.
	\end{abstract}
	
	\date{\today}
	\maketitle

	\section{Introduction}

	 The aim of this paper is to study the connection between bounds on the depth of tensor products of modules and the vanishing of Tor modules. This area traces its origin to the seminal work of Auslander \cite{A1}, who established that for an unramified regular local ring $R$, if $M \otimes_R N$ is torsion-free for some nonzero finitely generated modules $M$ and $N,$  then $M$ and $N$ are torsion-free, the pair $(M,N)$ is \textit{Tor-independent}, i.e., $Tor_i^R(M,N) =0 $ for all $i \geq 1,$ and further, the \textit{depth formula},  $depth_R( M \otimes_R N) = depth_R(M) + depth_R(N) - depth_R(R)$ holds. Later, Lichtenbaum \cite{Lic} generalized this result to the case of all regular local rings.  
	 
	 Huneke and Wiegand \cite[2.7]{HW} further extended these results to abstract hypersurface rings, showing that if $M \otimes_R N$ is reflexive, and either of these modules has constant rank, then the pair $(M,N)$ is Tor-independent and at least one of the $M$ and $N$ is reflexive (see  \cite[1.3]{CP}, \cite{HW3}). The depth formula holds in this case as well. Later, Huneke, Jorgensen, and Wiegand \cite{HJW} used quasilifting to establish Tor-vanishing results in codimensions 2 and 3 under additional conditions.
	 
	 Without any assumption on the modules or the ring, even totally reflexivity of $M \otimes_R N$ does not imply Tor-independence; see \cite[3.14]{C} for a counterexample over a codimension $2$ local complete intersection. 
	 
	 Recall that a local ring $(R,m)$ is a\textit{ complete intersection} if its $m$-adic completion is of the form $Q/(\underline{x})$ where $Q$ is a complete regular local ring and $\underline{x}$ is a regular sequence. Let $\lambda_R(M)$ denote the length of an $R$-module $M.$ In \cite{CST}, Celikbas, Sadeghi, and Takahashi proved the following 

\begin{theorem}\cite[1.2]{CST} \label{theorem_cst} Let $R$ be a local complete intersection ring, and let $M$ and $N$ be nonzero finitely generated $R$-modules satisfying:
	\begin{enumerate}[label=\normalfont(\alph*)]
		\item $\lambda(Tor_i^R(M,N)) < \infty$ for all $i \geq 1$ , and 
		\item $depth_R(M \otimes_R N)  \geq cx_R(M,N)+1.$
	\end{enumerate}
	Then \begin{enumerate}[label=\normalfont(\roman*)]
		\item $depth_R( M \otimes_R N) \geq depth_R(M) + depth_R(N)- depth_R(R),$ and 
		\item $depth_R( M \otimes_R N) = depth_R(M) + depth_R(N)- depth_R(R)$ if and only if the pair $(M,N)$ is Tor-independent. 
	\end{enumerate} 
\end{theorem}
	
	Recall that a ring $R$ is $\mathcal{TE}$ if for any pair $(M,N)$ of finite $R$-modules (see \cite[6.3]{AvBuch})
	\begin{align} \label{eqn_TE} 
		Tor_i^R(M,N) = 0  \text{ for all $i \gg 0$} \implies Ext^i_R(M,N) = 0 \text{ for all $i \gg 0.$}
	\end{align} 
	
 Avramov and Buchweitz \cite[III]{AvBuch} proved that local complete intersection rings are $\mathcal{TE}.$  Subsequent works by Huneke and Jorgensen \cite{HJ}, and Jorgensen and Sega \cite{JS} shows that we have strict inclusions $$\text{local complete intersection rings} \subsetneq \text{local $\mathcal{TE}$ rings} \subsetneq \text{local Gorenstein rings}.$$ 

More recently, Kimura, Lyle, Otake, and Takahashi \cite{KJYT} have studied some properties of local $\mathcal{TE}$ rings and their related homological cousins.

	Our main result is,
	\begin{theorem} \label{theorem_main_cst} 
		Let $R$ be a local ring, and let M and N be nonzero finitely generated $R$-modules. Assume that one of the following holds 	\begin{enumerate}[label=\normalfont(\alph*)]
			\item \label{item_TE} $R$ is a local complete $\mathcal{TE}$ ring with infinite residue field, $vpd_R (M) < \infty$ and $M$ is locally free on the punctured spectrum of $R.$ 
			\item \label{item_ci} $R$ is a local complete intersection ring and $\lambda(Tor_i^R(M,N)) < \infty$ for all $i \geq 1.$ 
		\end{enumerate} In addition, assume that $depth_R( M \otimes_R N) \geq cx_R(M,N) + 1.$

			Then \begin{enumerate}[label=\normalfont(\roman*)]
			\item $depth_R( M \otimes_R N) \geq depth_R(M) + depth_R(N)- depth_R(R),$ and 
			\item $depth_R( M \otimes_R N) = depth_R(M) + depth_R(N)- depth_R(R)$ if and only if the pair $(M,N)$ is Tor-independent. 
		\end{enumerate} 
	\end{theorem}

	The part \ref{item_ci} of Theorem \ref{theorem_main_cst} recovers Theorem \ref{theorem_cst} cited above, while part \ref{item_TE} is new. The proof of Theorem \ref{theorem_main_cst} is based on a quasilifting construction of Huneke, Jorgensen, and Wiegand \cite{HJW}, which allows ``lifting" of the setup to a pair of $Q$-modules where $Q$ is a quasi-deformation of $R.$ We then apply \cite[III]{AvBuch} and a result of Celikbas and Dao \cite[2.4]{CD} to obtain complexity reduction.

	One of the motivations behind this work was to answer a question \cite[3.10]{CST} which asked if one can replace the hypothesis ``$\lambda(Tor_i^R(M,N)) < \infty$ for all $i \geq 1$" with ``$\lambda(Tor_i^R(M,N)) < \infty$ for all $i \gg 0$" ? We construct an example (see Example \ref{eg_4}) to show that this is not true in this generality. However, with this relaxed hypothesis, we establish the following lower bound on the depth of tensor product:
	
	\begin{theorem}\label{theorem_syz} Let $R$ be a local complete intersection ring, and let $M$ and $N$ be nonzero finitely generated $R$-modules. Assume that 
		\begin{enumerate}[label=\normalfont(\alph*)]
			\item \label{item_a} $\lambda(Tor_i^R(M,N)) < \infty$ for all $i \gg 0.$
			\item $depth_R(M \otimes_R N) \geq cx_R(M,N) + 1.$
		\end{enumerate} Then $depth_R(M \otimes_R N) \geq syz_R(M) + syz_R(N) - depth_R(R).$  
	\end{theorem}
	
	Here by $syz_R(X),$ we mean the largest integer $t$ less than or equal to $depth_R(R),$ such that $X$ is a $t$'th syzygy module. 
	
	
  	We now describe the organization of this paper. Section \ref{section_two} contains notations and some preliminary results. In section \ref{section_approx}, we discuss some results concerning lifting of hulls and approximations. A proof of Theorem \ref{theorem_main_cst} is provided in section \ref{sec_main}. The last section is devoted to Example \ref{eg_4} and the proof of Theorem \ref{theorem_syz}.  
	
	\section{Preliminaries}  \label{section_two} 
	
	Throughout, all rings are commutative Noetherian rings, and modules are assumed to be finitely generated. For any $R$-module $X,$ we denote $Hom_R(X,R)$ by $X^*.$

\subsection{Approximations and hulls} \label{section_approx_hull} 	See \cite{Xu} and \cite{CPST}. Let $mod(R)$ be the category of finitely generated $R$-modules. Let $\cD$ and $\cE$ be any two subcategories of $mod(R).$  
 Following  \cite[2.1]{CPST}, we define \begin{align}
	\label{defn_perp}	\cD^{\perp} &= \{Y \in mod(R)\ | \ Ext^i_R(X,Y) = 0 \text{ for all } X \in \cD \text{ and all } i > 0\} 
\end{align} We say that a module $M \in mod(R)$ \textit{admits a $\cD$-approximation} if there are modules $X \in \cD$ and $Y \in \cD^{\perp}$ satisfying a short exact sequence \begin{align} \label{eqn_G_approx} 
0 \rar Y \rar X \rar M \rar 0.
\end{align} The definition \ref{defn_perp}  and the notion of $\cD$-approximation of a module  were studied in \cite{CPST}. It follows from  \cite{AM} that any $M \in mod(R)$ with finite $G\mbox{-}dim_R(M) < \infty,$ admits a $\cG(R)$-approximation where $\cG(R)$ is the full subcategory of $mod(R)$ of totally reflexive $R$-modules. 

Analogously, we define
\begin{align}
	\label{defn_top}	\cE^{\top} & =  \{X\in mod(R)\ | \ Ext^i_R(X,Y) = 0 \text{ for all } Y \in \cE \text{ and all } i > 0\}
\end{align} 

An \textit{$\cE$-hull of a module $M \in mod(R)$} is a short exact sequence \begin{align}
	0 \rar M \rar X \rar Y \rar 0 
\end{align} where $X \in \cE$ and $Y \in \cE^T.$ We say that $\cD$ is \textit{closed under extension} if for any short exact sequence of $R$-modules $0 \rar Y' \rar Y \rar Y'' \rar 0$ with $Y', Y'' \in \cD,$ it follows that $Y \in \cD.$ We note the following properties that are easy to verify. 

\begin{enumerate}[label=\normalfont(\alph*)]
	\item \label{item_a1} If $\cD' \subseteq \cD,$  then $\cD^{\perp} \subseteq \cD'^{\perp}$ and $\cD^{\top} \subseteq \cD'^{\top}.$ 
	\item \label{item_d} $\cD\subseteq \cD^{\perp \top}$ and $\cE\subseteq \cE^{\top \perp}$
	\item \label{item_e} $\cD^{\perp}$ and $\cE^{\top}$ are closed under extension.
	\item \label{item_c} If $0 \rar X \rar Y \rar Z \rar 0$ is a short exact sequence such that $X,Y \in \cD^{\perp},$ then $Z\in \cD^{\perp}.$
\end{enumerate}

		\subsection{Gorenstein dimension} \label{section_gorenstein} See  \cite{AB}, \cite{Av1},  and \cite{AGP}. We say that a finitely generated $R$-module $M$ is totally reflexive if the natural map $M \rar M^{**}$ is bijective, and $Ext^i_R(M,R) = Ext^i_R(M^*, R) = 0$ for all $i \geq 1.$ The smallest integer $n$ (if there is any) for which there exists an exact sequence $0 \rar G_n  \rar \cdots \rar G_{0} \rar M \rar 0,$ where each $G_i$ is totally reflexive, is called the \textit{Gorenstein dimension} of $M.$ We express this as $G\mbox{-}dim_R(M) = n.$ By convention, $G\mbox{-}dim_R(0)= -\infty.$ 
		
		Let $R$ be a local ring. Let $M$ be a finitely generated $R$-module. Assume $G\mbox{-}dim_R(M) < \infty.$  We recall some properties
		\begin{enumerate}[label=\normalfont(\alph*)]
			\item \label{item_proj_hull} (Projective hull) There is an exact sequence $0 \rar M \rar X \rar G \rar 0$ of $R$-modules such that $G\mbox{-}dim_R(G) = 0$ and $pd_R(X) = G\mbox{-}dim_R(M)$; see \cite[2.17]{CFH}.
			\item \label{item_AB} (Auslander-Bridger formula) $G\mbox{-}dim_R(M) + depth_R(M) = depth_R(R).$ In particular, $G\mbox{-}dim_R(M) \leq depth_R(R);$  see \cite{AB}.
		\end{enumerate}
		
		\subsection{Virtual projective dimension and complete intersection dimension} \label{section_gorenstein_com} See  \cite{Av1},  and \cite{AGP}. 
		Let $(R,m)$ be a local ring. A surjective map $Q \rar R'$ of rings such that its kernel is generated by a length $c,$ $Q$-regular sequence, is called a (codimension $c$) \textit{deformation} of $R'.$ A diagram of local ring maps $R \rar R' \leftarrow Q$ such that $Q \rar R'$ is a deformation and $R \rar R'$ is faithfully flat is called a \textit{quasi-deformation} of $R.$ The \textit{complete intersection dimension} of $M$ is 
		$$CI\mbox{-}dim_R(M) = \inf \left\{pd_Q(M \otimes_R R') - pd_Q(R')| R \rar R' \leftarrow Q \text{ is a quasi-deformation}\right\}.$$ 
		
		Let $R$ be a local ring with infinite residue field. For such rings, the \textit{virtual projective dimension} \cite[3.3]{Av1} of a module $M$ is  $$vpd_R(M) = \min \{pd_P(M \otimes_{R} \widehat{R})\ | \ P \text{ is a deformation of } {\widehat{R}}\},$$ where $\widehat{R}$ is the $m$-adic completion of $R.$ 
	
		Let $M$ and $N$ be finitely generated $R$-module such that $CI\mbox{-}dim_R(M) < \infty.$ 
		\begin{enumerate}[label=\normalfont(\roman*)]
		\item \label{item_tor_g_dim_0} If $Tor_i^R(M,N) = 0$ for all $i \gg 0,$ then $Tor_i^R(M,N) = 0$ for all $i \geq CI\mbox{-}dim_R(M) + 1.$ Similarly, if $Ext_R^i(M,N) = 0$ for all $i \gg 0,$ then $Ext_R^i(M,N) = 0$ for all $i \geq CI\mbox{-}dim_R(M) + 1$ \cite[4.7, and 4.9]{AvBuch}. 
	
		\item   \label{item_ci_dim_0} Let $G$ be an $R$-module such that  $CI\mbox{-}dim_R(G) = 0.$ Let $N$ be an $R$-module such that $\lambda(Tor_i^R(G,N)) < \infty$ for all $i \gg 0.$ Since localization doesn't increase the complete intersection dimension (see \cite[1.6]{AGP}), we conclude from \ref{item_tor_g_dim_0}  that $\lambda(Tor_i^R(G,N)) < \infty$ for all $i \geq 1.$
		\item \label{item_sather} If $0 \rar X' \rar X \rar X'' \rar 0$ is an exact sequence of $R$-modules and one of the modules has finite complete intersection dimension while one of the remaining two has finite projective dimension, then the third module has finite $CI$ dimension \cite[3.6]{Sather}.
		\item\label{item_AvBuch} Let $R$ be a local complete intersection ring and let $M$ and $N$ be finitely generated $R$-modules $M$ and $N.$ Then it follows by \cite[III]{AvBuch}  that $$Ext^i_R(M,N) = 0 \text{ for } i \gg 0 \iff Tor_i^R(M,N) = 0 \text{ for all } i \gg 0.$$ 
		\item \label{item_vpd_finite}    If $vpd_R(M) < \infty,$ then $CI\mbox{-}dim_R(M) < \infty;$ see \cite[5.11]{AGP}. 
	\end{enumerate}
	
	\subsection{A long exact sequence} Let $R = Q/(x)$ where $Q$ is a local ring and $x$ is a non zero-divisor on $Q.$ Given finitely generated $R$-modules $M$ and $N,$ the change of ring spectral sequence  \cite[5.6.6]{W}  collapses to a long exact sequence 
	\begin{align} \label{seq_change}
		\cdots	\rar Tor_3^Q(M,N)\rar Tor_3^R(M,N) \rar Tor_1^R(M,N) \rar Tor_2^Q(M,N) \rar \nonumber  \\
		Tor_2^R(M,N) \xrightarrow{\Psi}  Tor_0^R(M,N) \rar Tor_1^Q(M,N) \rar Tor_1^R(M,N) \rar 0.
	\end{align}

	\subsection{Complexity} \label{sec_cmpx} The complexity of a sequence of non-negative integers $B = \{b_i\}_{i \geq 0}$ is defined as \cite{Av1} $$cx(B) = inf \{r \in \N \cup \{0\} \ |\ b_n \leq \alpha \cdot n^{r-1} \text{ for some  $\alpha  \in \R$ and for all } n\gg 0\}.$$ Following \cite{AvBuch},  the complexity $cx_R(M,N)$ of a pair $(M,N)$  of $R$-modules is defined as $$cx_R(M,N) = cx\left(\nu_R(Ext^i_R(M,N))\right),$$ where $\nu_R(X)$ denotes the minimal number of generators of a finitely generated $R$-module $X.$ 
%

	\subsection{The depth formula} \label{section_depth_formula} We say that two finite $R$-modules $M$ and $N$ satisfy the \textit{depth formula} if $depth_R(M) + depth_R(N)= depth_R(M \otimes_R N) + depth_R(R).$ Auslander \cite{A1} proved it for Tor-independent pairs where one of the modules has finite projective dimension. We refer to  \cite{I} or \cite{ArYos}  for generalization to the case where one of the modules has finite complete intersection dimension. For some recent generalizations, see \cite{BerJo2,CJ,KJYT}.

\section{Approximations and Hulls}  \label{section_approx}

Let $R$ be a ring and let $mod(R)$ be the category of finitely generated $R$-modules. Let  $\cG(R)$ and $\wtil{\cG}(R)$ be the full subcategories of totally reflexive $R$-modules and $R$-modules with finite Gorenstein dimension, respectively. Let $\cP(R)$ and $\wtil{P}(R)$  be the full subcategories of finitely generated projective $R$-modules and finitely generated $R$-modules with finite projective dimension, respectively. Recall that we have (see, for instance, \cite[23]{Masek}) \begin{align} \label{eqn_approx_hull}
\cG(R) =	{\cP}(R)^{\top} \cap \wtil{\cG}(R) = 	\wtil{\cP}(R)^{\top} \cap \wtil{\cG}(R).
\end{align} Furthermore, if $R$ is a Gorenstein ring, then $\cG(R) = {\cP}(R)^{\top}  =  \wtil{\cP}(R)^{\top}.$ The result that follows shows an analogous equality corresponding to the $\perp$ operation. 

\begin{lemma} \label{example_1} 
With notation as above, we have $$\cG(R)^{\perp} \cap \wtil{\cG}(R) = \wtil{\cP}(R).$$
In particular, if $R$ is a Gorenstein ring, then $\cG(R)^{\perp} = \wtil{\cP}(R).$
\end{lemma}
\begin{proof}
	 It follows from \cite[21,23]{Masek} that $\wtil{\cP}(R) \subseteq \cG(R)^{\perp} \cap \wtil{\cG}(R).$ Conversely, suppose $X \in \cG(R)^{\perp}\cap \wtil{\cG}(R).$ Let $F_0 \rar X$ be a surjection onto $X$ from a free $R$-module $F_0.$ Set $\Omega^1_RX = Ker(F_0 \rar X).$ For any $G \in \cG(R),$ the Auslander's transpose $Tr_R(G^*) \in \cG(R);$ see \cite[4.9]{AB}. In particular, $Ext^i_R(Tr_R(G^*), X) = 0 \text{ for } i = 1,2.$ Thus, by the 4-term sequence in \cite[2.8]{AB}, the natural map $G^* \otimes_R X \rar Hom_R(G^{**}, X)$ is an isomorphism. 
	
	Therefore, it follows from the sequence $0 \rar \Omega^1_R X \rar F_0 \rar X \rar 0$ that the map $Hom(G^{**}, F_0) \rar Hom(G^{**}, X)$ is surjective. Since $G$ is totally reflexive, we get that $Ext^1_R(G, \Omega^1_RX) = 0.$ The long exact Hom sequence shows that $\Omega^1_RX \in \cG(R)^{\perp} \cap \wtil{\cG}(R).$ 
	
	Inductively, we conclude $\Omega^r_RX \in \cG(R)^{\perp} \cap \wtil{\cG}(R)$ for all $r \geq 1.$ Taking $r$ large enough, we may assume that $\Omega^r_RX$ is a totally reflexive $R$-module. Consider the short exact sequence $0 \rar \Omega^{r+1}_RX \rar F_r \rar \Omega^r_RX \rar 0$ for $r$ sufficiently large. Since $\Omega^r_RX \in \cG(R) \cap \cG(R)^{\perp},$ we get $Ext^1_R(\Omega^r_RX, \Omega^{r+1}_RX) = 0,$ which proves that $\Omega^r_RX$ is projective, i.e. $X \in \wtil{P}(R).$ 
	
	The last assertion follows from the fact that for a Gorenstein ring, $\wtil{\cG}(R) = mod(R);$ see \cite[4.15, 4.20]{AB}. 
\end{proof}

Let $\varphi: Q \rar R$ be a homomorphism of rings which induces a finitely generated $Q$-module structure on $R.$ In particular, any module $M \in mod(R)$ can also be viewed as a module in $mod(Q).$ For a subcategory $\cC$ of $mod(R),$ we will use $\cC_Q,$ when we want to consider the modules in $\cC$ as $Q$-modules. For instance, $(\cC_Q)^{\top} =  \{X\in mod(Q)\ | \ Ext^i_Q(X,Y) = 0 \text{ for all } Y \in \cC \text{ and all } i > 0\}.$ 

\begin{lemma} \label{example_2}  
Let $Q$ be a Gorenstein local ring. Let $Q \rar R$ be a homomorphism of local rings inducing a finitely generated $Q$-module structure on $R.$ Assume that $pd_Q(R) < \infty.$ Then $$\cG(Q) = ({\cP}(R)_Q)^{\top} = (\wtil{\cP}(R)_Q)^{\top}.$$ 
\end{lemma}
\begin{proof} Since ${\cP}(R)_Q \subseteq  \wtil{\cP}(Q),$ by \ref{section_approx_hull}\ref{item_a1}, we have an inclusion $$\cG(Q) = \wtil{\cP}(Q)^{\top}  \subseteq (\wtil{\cP}(R)_Q)^{\top},$$ where the first equality follows from equation \eqref{eqn_approx_hull} and the assumption on $Q.$
	
Conversely, suppose $M$ be a $Q$-module  $({\cP}(R)_Q)^{\top},$ i.e., $Ext^i_Q(M,R) = 0$ for $i > 0.$ Since $pd_Q(R) < \infty,$ by \cite[3.1.25]{BH}, $inj\mbox{-}dim_Q(R) < \infty.$ By a result of Ischebeck \cite[3.1.24]{BH}, we get $depth_Q (M)= depth_Q (Q),$ that is, $M \in \cG(Q).$ 
	
\end{proof}

	
	

	
	


Let $\cB$ be a subcategory of $mod(Q)$ and $\cC$ be a subcategory of $mod(R).$ We say that an $R$-module $M$ admits a $\cB$-\textit{approximation} or $\cB$-\textit{hull} if it does so as a $Q$-module.  We say that $\cC$ admits \textit{ $\cB$-approximation} (resp., \textit{ $\cB$-hull}) if any $M \in \cC$ admits a $\cB$-approximation (resp., { $\cB$-hull}).  The following result explores some general conditions under which approximation/hulls can be {lifted} from $R$ to $Q.$  

\begin{proposition} \label{prop_approx_hull} 
	
	Let $Q \rar R$ be a ring homomorphism that makes $R$ a finitely generated $Q$-module. Let $M$ be a finitely generated $R$-module. Let  $\cB \subseteq mod(Q)$ be a subcategory of $mod(Q).$ Let $\cC$ and $\cD$ be subcategories of $mod(R).$
	\begin{enumerate}[label=\normalfont(\alph*)]
		\item \label{item_hull} Assume that $\cB \subseteq (\cC_Q)^{\top}.$  If $(\cC_Q)^{\top}$ admits $\cB$-approximation, then any $\cC$-hull  of $M$ determines a $\cB^{\perp}$-hull of $M$ as a $Q$-module.
		\item  \label{item_approx} Assume that $(\cD_Q)^{\perp} \subseteq \cB^{\perp}.$  If  $\cD_Q$ admits $\cB$-approximation (resp. $\cB^{\perp}$-hull), then any $\cD$-approximation of $M$ on $R$ determines a $\cB$-approximation (resp. $\cB^{\perp}$-hull) of $M$ on $Q.$
	\end{enumerate}  
\end{proposition}
We note that by \ref{section_approx_hull}\ref{item_a1}, \ref{section_approx_hull}\ref{item_d},  and the hypothesis assumed in part \ref{item_hull}, it follows that \begin{align} \label{eqn_C_contained_in_B_perp} 
	\cC_Q \subseteq \cB^{\perp}.
\end{align} Thus, Proposition \ref{prop_approx_hull}\ref{item_hull} gives a larger class for which a hull of $M$ exists. 
\begin{proof}[Proof of proposition \ref{prop_approx_hull}]
	
	Let $0 \rar M \rar X \rar G \rar 0$ be a $\cC$-hull of $M.$ By construction, $G \in \cC^{\top}$ and  by  \ref{item_hull}, it admits a $\cB$-approximation $0 \rar Y' \rar G' \rar  G \rar 0.$ Consider the following pullback diagram
	\begin{equation} \label{dgm_lift_hull}
		\begin{gathered}
			\xymatrix@C-=4.2cm@R-=2.2cm{
				& & Y' \ar@{=}[r]\ar[d] & Y' \ar[d] & & \\ 0 \ar[r] & M \ar[r]\ar@{=}[d] & X' \ar[d] \ar[r] & G' \ar[d] \ar[r] & 0 \\ 0 \ar[r] & M \ar[r] & X \ar[r] & G \ar[r] & 0 }
		\end{gathered} 
	\end{equation} By \eqref{eqn_C_contained_in_B_perp}, $X \in \cB^{\perp}.$ By closure under extension \ref{section_approx_hull}\ref{item_e}, it follows that  $X' \in \cB^{\perp}.$ Furthermore, $G' \in \cB \subseteq \cB^{\perp \top}.$ Thus, \ref{item_hull} follows by observing that the exact sequence $0 \rar M \rar X' \rar G' \rar 0$ is a $\cB^{\perp}$-hull of $M$ on $Q.$
	
	Now, let $0 \rar X \rar G \rar M \rar 0$ be a $\cD$-approximation of $M.$ Here $G \in \cD,$ and $X \in \cD^{\perp}.$  By the first hypothesis assumed  in \ref{item_approx}, $G$ admits a $\cB$-approximation $0 \rar X' \rar G' \rar  G \rar 0,$ where $G' \in \cB,$ and $X' \in \cB^{\perp}.$ Consider the following pullback diagram
	\begin{equation} \label{dgm_lift_approximation}
		\begin{gathered}
			\xymatrix@C-=4.2cm@R-=2.2cm{
				& X'  \ar@{=}[r]\ar[d] & X' \ar[d] & & \\ 0 \ar[r] & Y' \ar[r]\ar[d] & G' \ar[d] \ar[r] & M\ar@{=}[d] \ar[r] & 0 \\ 0 \ar[r] & X \ar[r] & G \ar[r] & M \ar[r] & 0 }
		\end{gathered} 
	\end{equation} Since $X \in \cD^{\perp} \subseteq \cB^{\perp}$ and $X' \in \cB^{\perp},$ by closure under extension \ref{section_approx_hull}\ref{item_e}, $Y' \in \cB^{\perp}.$ Thus, the  sequence $0\rar  Y' \rar G' \rar M \rar 0$ is a $\cB$-approximation of $M,$ which proves the first claim of \ref{item_approx}. 
	
	To prove the second claim in \ref{item_approx}, consider the following diagram, where the middle vertical sequence is a $\cB^{\perp}$-hull of $G \in \cD.$ 
	\begin{equation} \label{dgm_lift_approximation_1}
		\begin{gathered}
			\xymatrix@C-=4.2cm@R-=2.2cm{
				0 \ar[r] & X \ar@{=}[d]\ar[r] & G \ar[d] \ar[r] & M \ar[r] \ar[d] & 0 \\ 0 \ar[r] & X \ar[r]& X' \ar[d] \ar[r] & Y'\ar[d] \ar[r] & 0 \\  & & G' \ar@{=}[r] & G'&}
		\end{gathered} 
	\end{equation} Here  $Y' = coker(X \rar X').$  In the middle horizontal sequence, we have $X \in (\cD_Q)^{\perp} \subseteq \cB^{\perp}$ and $X' \in \cB^{\perp},$ by \ref{section_approx_hull}\ref{item_c}, $Y' \in \cB^{\perp}.$ Thus, the right vertical sequence is a $\cB^{\perp}$-hull of $M.$ 
\end{proof}

\begin{exmp} Let $R$ be a Gorenstein local ring such that $R = Q/(\underline{x})$ where $Q$ is a local ring and $\underline{x}$ is a regular sequence on $Q.$ Set $\cC = \wtil{\cP}(R),$  $\cD = \cG(R),$ and $\cB = \cG(Q).$ By Lemmas \ref{example_1} and \ref{example_2}, $$\cB = (\cC_Q)^{\top},\ \ \cD^{\perp} = \wtil{\cP}(R) \subseteq \cB^{\perp} = \wtil{\cP}(Q),\ \text{ and } \cC^{\top}= \cG(R) = \cD.$$ Thus, by Proposition \ref{prop_approx_hull}, any projective hull of $M$ on $R$ determines a projective hull of $M$ on $Q.$ Similarly, any $\cG(R)$-approximation of $M$ determines a $\cG(Q)$-approximation and a projective hull of $M$ on $Q.$
\end{exmp}

	  Recall that any finite $R$-module with $G\mbox{-}dim_R(M) < \infty$ admits a projective hull and a $\cG(R)$-approximation; see \ref{section_gorenstein}\ref{item_proj_hull} and \ref{section_approx_hull}\eqref{eqn_G_approx}. We need the following result, which shows that any projective hull of a finitely generated module $M$ on a ring $R = Q/(x),$ determines a $\cG(Q)$-approximation of $M$ on $Q.$ This approximation is used in the inductive step of Theorem \ref{theorem_main_cst}. 
	
	\begin{proposition}\label{prop_projective_hull} 	Let $(Q,n)$ be a local ring and let $x \in n$ be a regular element. Let $R = Q/(x)$ and let $M$ be finitely generated $R$-module with $G\mbox{-}dim_Q(M) < \infty.$ Let  $$0 \rar M \rar X \rar G \rar 0$$ be a projective hull of $M$ on $R.$ Then there exist $Q$-modules $\wtil{G}$ and $\wtil{X}$ such that $$0 \rar \wtil{X} \rar \wtil{G} \rar M \rar 0$$ is a $\cG$-approximation of $M$ on $Q,$ and $pd_Q(\wtil{X}) = G\mbox{-}dim_R(M).$  Moreover, if $M$ has locally finite projective dimension on the punctured spectrum of $R$ then $\wtil{G}$ has locally finite projective dimension on the punctured spectrum of $Q.$ 
	\end{proposition}
	\begin{proof} By construction, $G\mbox{-}dim_R(G) = 0$ and $pd_R(X) = G\mbox{-}dim_R(M).$ Let $F_0 \onto X$ be a minimal surjection onto $X$ where $F_0$ is a free $R$-module. Let $\wtil{F_0}$ be a lifting of $F_0$ to $Q.$ Set $\wtil{G} = Ker(\varphi)$ where $\varphi: \wtil{F_0} \onto G$ is the natural map commuting through $X.$ Then the projective hull of $M$ fits in the following commutative diagram:
		\begin{equation} \label{dgm_m}
			\begin{gathered}
				\xymatrix@C-=4.2cm@R-=2.2cm{
					& \wtil{X} \ar@{=}[r]\ar[d] & \wtil{X} \ar[d] & & \\ 0 \ar[r] & \wtil{G} \ar[r]\ar[d] & \wtil{F_0} \ar[d] \ar[r]^{\varphi} & G \ar@{=}[d] \ar[r] & 0 \\ 0 \ar[r] & M \ar[r] & X \ar[r] & G \ar[r] & 0 }
			\end{gathered} 
		\end{equation} where $\wtil{X} = Ker(\wtil{F_0} \rar X).$ It is clear that $pd_Q(\wtil{X}) < \infty.$ The Auslander-Bridger formula \ref{section_gorenstein}\ref{item_AB} shows that $$pd_Q(\wtil{X}) =  depth_Q (Q) - depth_Q (\wtil{X}) = depth_R(R) - depth_R(M)  = G\mbox{-}dim_R(M),$$ where the second equality follows from the middle verticle sequence and the equality $depth_R(X) = depth_R(M);$ see \ref{section_gorenstein}\ref{item_proj_hull}. Since $G\mbox{-}dim_Q(M)$ is assumed to be finite, by \cite[18]{Masek}, the left vertical sequence implies that $G\mbox{-}dim_Q(\wtil{G})  < \infty.$ 
		The depth lemma shows that $$depth_Q (\wtil{G}) \geq \min(depth_Q (\wtil{F_0}), depth_Q (G) + 1) = depth_Q (Q).$$ Now, another application of Auslander-Bridger formula proves that $\wtil{G}$ is a totally reflexive $Q$-module thus completing the proof of the first part.
		
		Let $p \in Spec\ Q - \{n\}$ be a non-maximal prime. Localizing the leftmost vertical exact sequence \ref{dgm_m} at $p$, we get \begin{align*}
			0 \rar \wtil{X}_{p} \rar \wtil{G}_p \rar M_p \rar 0.
		\end{align*}
		 If $x \in p$ and if $\overline{p}$ is the image of $p$ in $R,$ then $M_p \cong M_{\overline{p}}$ (as a $Q_p$-module). Since $pd_{R_{\overline{p}}}(M_{\overline{p}}) < \infty,$ we get $pd_{Q_p} (M_{\overline{p}}) < \infty.$ If $x \notin p,$ then $M_p = 0,$ thus $\wtil{G}_p \cong \wtil{X}_p.$ Thus, in either case, $pd_{Q_p} (\wtil{G}_p) < \infty,$  which completes the proof of the last assertion.

	\end{proof}

	\section{Proof of Theorem \ref{theorem_main_cst}}  \label{sec_main} 
    In this section, we will prove Theorem \ref{theorem_main_cst} and place it in the context of a more complete result in  Theorem \ref{theorem_gen}. 
  	
	 The result below was proved in \cite[3.6]{CST} using Araya and Yoshino's depth formula \cite{ArYos}. We provide a different proof based on Peskine and Szpiro's acyclicity lemma. 
	\begin{lemma}\cite[3.6]{CST}   \label{prop_depth_acyclicity} Let $R$ be a local ring, and let $M$ and $N$ be finitely generated $R$-modules such that $CI\mbox{-}dim_R(M) < \infty.$ Set $g = CI\mbox{-}dim_R(M).$ Assume the following: 
		\begin{enumerate}[label=\normalfont(\roman*)]
			\item $depth_R(N) \geq  g.$ 
			\item for some positive integer $r \leq g$ and for all $i = r, \ldots, g,$ either $depth_R(Tor_i^R(M,N)) = 0$ or $Tor_i^R(M,N) = 0.$
		\end{enumerate} Then $Tor_i^R(M,N) = 0$ for all $i \gg 0$ if and only if $Tor_i^R(M,N) = 0$ for all $i \geq r.$  
	\end{lemma}
	\begin{proof} It is sufficient to prove the forward direction. So assume that $Tor_i^R(M,N) = 0$ for all $i \gg 0.$ Consider a projective hull of $M,$  \begin{align} \label{M_proj_hull} 
			0 \rar M \rar X \rar G \rar 0.
		\end{align} Here $CI\mbox{-}dim_R(G) = 0$ and $pd_R(X) = CI\mbox{-}dim_R(M) = g;$ see  \ref{section_gorenstein}\ref{item_proj_hull} and \ref{section_gorenstein_com}\ref{item_sather}. We tensor \eqref{M_proj_hull} with $N,$ to obtain $Tor_i^R(G, N) = 0$ for all $i \gg 0,$ which by \ref{section_gorenstein_com}\ref{item_tor_g_dim_0} gives $Tor_i^R(G,N) = 0$ for all $i \geq 1.$  This means that $depth_R(Tor_i^R(X,N)) \in  \{0, \infty\}$ for $i = r,\ldots, g.$ Set $Y = \Omega_R^{r-1}X.$ Then $pd_R(Y) = g -r+1$ and by assumption, for each $i = 1, \ldots, g-r+1,$ either $depth_R(Tor_i^R(Y,N)) = 0$ or $Tor_i^R(Y,N) = 0.$ Let $F_{\ast}  \rar Y \rar 0$ be a free resolution of $Y.$ Consider the complex $$\mathbb{F}(N): \ \ 0 \rar F_{g-r+1} \otimes N \rar F_{s-r} \otimes N \rar \cdots \rar F_0 \otimes N \rar 0.$$ For $i = 1,\ldots, g-r+1,$ we have 
		\begin{enumerate}[label=\normalfont(\alph*)]
			\item $depth_R(F_i \otimes N) = depth_R(N) \geq g \geq g-r+1 \geq i.$
			\item $depth_R(H_i(\mathbb{F}(N))) = depth_R(Tor_i^R(Y,N)) = 0$ or $ H_i(\mathbb{F}(N))  = 0.$
		\end{enumerate} By acyclicity lemma \cite[1.4.24]{BH} the complex $\mathbb{F}(N)$ is exact. Hence $Tor_i^R(Y,N) = 0$ for $i = 1,\ldots, g-r+1.$ Since $pd_R(Y) =g-r+1,$ we get $Tor_i^R(Y,N) = 0$ for all $i \geq 1.$ Equivalently, $Tor_i^R(M,N) \cong Tor_i^R(X,N) = 0$ for all $i \geq r.$ 
	\end{proof}

Following \cite[6.3]{AvBuch}, we say that a ring is $\mathcal{ET}$ if for all finitely generated $R$-mod $M,N,$ \begin{align} \label{eqn_ET} 
	Ext^i_R(M,N) = 0 \text{ for all $i \gg 0$} \implies Tor_i^R(M,N) = 0 \text{ for all $i \gg 0.$}
\end{align}  Recall that a \textit{codimension $c$ {deformation}} of a local ring $R$ is a surjective map $Q \rar R$ of local rings with kernel generated by a $Q$-regular sequence of length $c.$ We recall two facts:
\begin{enumerate}[label=\normalfont(\roman*)]
	\item The local ring $R$ is Gorenstein if and only if its deformation $Q$ is Gorenstein; see \cite[3.1.19]{BH}.
	\item The local ring $R$ is $\mathcal{TE}$ if and only if it is Gorenstein and $\mathcal{ET};$ see  \cite[4.3]{JS}.
\end{enumerate}

It has come to our notice that the $\mathcal{TE}$ case of the following result was also proved in \cite[2.5(3)]{KJYT}.

\begin{lemma}\label{lemma_deformation_TE_iff_TE}
	Let $R$ be a local ring and let $Q \rar R$ be a codimension $c$ deformation. 
	\begin{enumerate}[label=\normalfont(\alph*)]
		\item If $R$ is a $\mathcal{TE}$ ring, then $Q$ is also a $\mathcal{TE}$ ring.  
		\item 	If $R$ is a local Gorenstein $\mathcal{ET}$ ring, then $Q$ is also a local Gorenstein $\mathcal{ET}$ ring. 
	\end{enumerate}

\end{lemma} 
\begin{proof} We start with a proof of the second assertion.
	
Assume that $R$ is a local Gorenstein $\mathcal{ET}$ ring. 
Inductively, it is sufficient to prove for the case when $c = 1.$ Let $x$ be a regular element in $Q$ such that $R = Q/(x).$ Let $M$ and $N$ be $Q$-modules such that $Ext^i_Q(M,N) =0$ for all $i \gg 0.$ Without loss of generality, assume that $pd_Q(M) = pd_Q(N) = \infty.$ Set $M_1 = \Omega_R^rM$ and $N_1 = \Omega_R^rN$ for $r$ large enough such that $M_1$ and $N_1$ are totally reflexive \cite[2.1.26]{BH}. Then $x$ is regular on $M_1$ and $N_1.$ Moreover, $Ext^i_Q(M_1,N_1) = 0$ for all $i \gg 0.$ Applying $Hom_Q(-, N_1)$ to the short exact sequence $0 \rar M_1 \xrightarrow{x} M_1 \rar M_1/xM_1 \rar 0,$  gives $Ext^i_Q(M_1/xM_1, N_1)= 0$ for all $i \gg 0.$ 

By a result of Rees \cite[3.1.16]{BH}, $Ext^{i+1}_Q(M_1/xM_1,N_1) \cong Ext^{i}_R(M_1/xM_1, N_1/xN_1),$ thus $Ext^i_R(M_1/xM_1, N_1/xN_1) = 0$ for all $i \gg 0.$ Since $R$ is $\mathcal{ET},$ we get $$Tor_i^R(M_1/xM_1, N_1/xN_1) = 0 \text{ for } i \gg 0.$$ By the change of ring sequence \ref{seq_change}, $Tor_i^Q(M_1/xM_1, N_1/xN_1) = 0 \text{ for } i \gg 0.$ Tensoring the short exact sequence $0 \rar M_1 \xrightarrow{x} M_1 \rar M_1/xM_1 \rar 0$ with $N_1/xN_1,$ we conclude that $Tor_i^Q(M_1, N_1/xN_1) = 0$ for all $i \gg 0.$ A similar argument shows that $Tor_i^Q(M_1,N_1) = 0$ for all $i \gg 0$ which is same as $Tor_i^Q(M,N) = 0$ for all $i \gg 0.$ This shows that $Q$ is also an $\mathcal{ET}$ ring. 

For the first assertion, we note the following implications:
\begin{align*}
R \text{  is a local $\mathcal{TE}$ ring}  \iff & R \text{ is a local Gorenstein $\mathcal{TE}$ ring} \\ \implies & Q \text{ is a local Gorenstein $\mathcal{TE}$ ring} \\ \iff & Q \text{  is a local $\mathcal{TE}$ ring},
\end{align*} where the bi-implications are by \cite[4.3]{JS}, and the implication follows from the first assertion above. 
\end{proof}

\begin{lemma} \label{lemma_existence_ring} 
	Let $(R,m)$ be a local complete $\mathcal{T}\mathcal{E}$ ring with infinite residue field and let $M,N$ be $R$-modules such that $cx_R(M,N) \geq 1.$ Assume that that $vpd_R(M) <\infty$ and $M$ has locally finite projective dimension over the punctured spectrum of $R.$ 
	
	Then there is a local complete $\cT \cE$ ring $Q$ and a non-zero divisor $x\in Q$ such that 
	\begin{enumerate}[label=\normalfont(\roman*)]
		\item ${R} = Q/(x),$
		\item $vpd_{Q} (M) < \infty,$ and 
		\item \label{item_cx}  $cx_{Q}({M}, {N}) = cx_R(M,N) -1.$ 
	\end{enumerate} 
\end{lemma}  
\begin{proof} As $M$ is finitely generated, $M \cong \widehat{M}.$ There is a  deformation $\gamma: P \rar R$ such that $vpd_R(M) = pd_P({M}).$ It is clear that the residue field of $P$ is also infinite. The kernel of $\gamma$ is generated by a regular sequence $x_1, \ldots, x_r.$ The hypothesis on $M$ implies that $\lambda(Ext^i_R(M,N)) < \infty$ for all $i \gg 0.$ Set $c = cx_R(M,N) \leq r.$ As in the proof of \cite[2.4]{CD}, the map $P \rar R$ factors through a local ring $(Q,n)$ such that $R = Q/(x)$ where $x = x_c,$ $cx_{Q}(M,N) = cx_{{R}}({M},{N}) -1,$ and  $vpd_{Q}({M}) < \infty.$ By taking completion, if necessary, we can assume $Q$ is complete. By Lemma \ref{lemma_deformation_TE_iff_TE}, $Q$ is a local $\cT \cE$ ring.  
\end{proof}

\begin{lemma} \label{lemma_existence_ring2} 
	Let $(R,m)$ be a local complete intersection ring, and let $M,N$ be $R$-modules such that $cx_R(M,N) \geq 1.$ Assume that  $\lambda(Tor_i^R(M,N)) < \infty$ for $i \gg 0.$ Then there is a local complete intersection ring $Q$ and a non-zero divisor $x\in Q$ such that 
	\begin{enumerate}[label=\normalfont(\roman*)]
		\item $\widehat{R} = Q/(x),$ 
		\item \label{item_cx2}  $cx_{Q}(\widehat{M}, \widehat{N}) = cx_R(M,N) -1.$ 
	\end{enumerate} where $\widehat{R}$ is the $m$-adic completion of $R$ and for any $R$-module $X,$ the $m$-adic completion of $X$ is denoted $\widehat{X}.$
\end{lemma}  
\begin{proof} If the residue field of $R$ is not infinite, we replace $R$ by the faithfully flat extension $R[x]_{mR[x]}.$ Thus, we may assume without loss of generality that the residue field of $R$ is infinite.

	Recall that the property of being a local complete intersection ring is preserved under localization. Thus, by \ref{section_gorenstein_com}\ref{item_AvBuch}, the hypothesis on the length of $Tor$s gives  that $\lambda(Ext^i_R(M,N)) < \infty$ for all $i \gg 0.$ By definition, there is a quasi-deformation $R \rar \widehat{R}\leftarrow P$ where $\widehat{R}$ is the $m$-adic completion of $R$ and $P$ is a complete regular local ring. Applying completion gives $\lambda(Ext^i_{\widehat{R}}(\widehat{M}, \widehat{N})) < \infty$ for all $i \gg 0.$ Here $pd_P(\widehat{M}) < \infty,$ and $cx_{\widehat{R}}(\widehat{M}, \widehat{N}) = cx_R(M,N);$ see \cite[4.1.1]{AvBuch}. Now, we apply \cite[2.4]{CD} to the pair $(\widehat{M}, \widehat{N})$ to obtain a local complete intersection ring $Q = P/(x_1,\dots,\widehat{x_c}, \dots, x_{r})$ which satisfies the above properties. Rest of the proof is same as in Lemma \ref{lemma_existence_ring}.
\end{proof}

	The result that follows uses the approximation obtained in Proposition \ref{prop_projective_hull}. We will use this to inductively reduce complexity in the proof of Theorem \ref{theorem_main_cst}. 
	
	\begin{lemma}\label{lemma_change_of_ring_tor_vanish_geq_0} Let $R = Q/(x)$ where $Q$ is a local ring and $x$ is a non zero-divisor on $Q.$ Let $M$ and $N$ be finitely generated $R$-modules such that $CI\mbox{-}dim_Q(M) < \infty$ and $\lambda(Tor_i^R(M,N)) < \infty$ for all $i \geq 1.$  Assume further that for some positive integer $n$ the following holds:
		\begin{enumerate}[label=\normalfont(\alph*)]
			\item \label{k_a} $depth_R(N) \geq CI\mbox{-}dim_R(M) + n.$ 
			\item \label{k_b} $depth_R(M \otimes_R N) \geq n + 1.$ 
		\end{enumerate} 
	Then there exists a $\cG(Q)$-approximation $0 \rar \wtil{X} \rar \wtil{G} \rar M \rar 0$  of $M$ such that
		\begin{enumerate}[label=\normalfont(\roman*)]
			\item \label{l_0} $CI\mbox{-}dim_Q(\wtil{G}) = 0,$ and $pd_Q(\wtil{X}) = CI\mbox{-}dim_R(M).$   
			\item \label{l_1} $\lambda(Tor_i^Q(\wtil{G},N)) < \infty$ for all $i \geq 1.$ 
			\item \label{l_2} $Tor_i^Q(\wtil{G},N) \cong Tor_i^Q(M,N)$ for all $i \geq 2$ and $Tor_1^Q(\wtil{G},N) \cong Tor_1^R(M,N).$   
			\item \label{l_4} $depth_Q (\wtil{G} \otimes_Q N)  \geq n.$
			
		\end{enumerate} Furthermore, if $M$ is locally free on the punctured spectrum of $R,$ then $\wtil{G}$ is locally free on the punctured spectrum of $Q.$ 
	\end{lemma}
	\begin{proof} We first note that $G\mbox{-}dim_R(M) < \infty$ as $CI\mbox{-}dim_R(M) < \infty;$ see \cite[1.4]{AGP}. consider a projective hull  of $M$ on $R,$ \ref{section_gorenstein}\ref{item_proj_hull}  
		\begin{align} \label{proj_hull} 
			0 \rar M \rar X \rar G \rar 0 
		\end{align} 
	
	Tensoring \ref{proj_hull} with $N$, it follows from the hypothesis on $Tor$s that $\lambda(Tor_i^R(G,N)) < \infty \text{ for } i \gg 0.$ Thus, by \ref{section_gorenstein_com} \ref{item_ci_dim_0} \begin{align} \label{l_a}
		\lambda(Tor_i^R(G,N)) < \infty  \text{ for } i \geq 1.
	\end{align} 

	The  hypothesis on $Tor_i^R(M,N)$ and \ref{l_a} imply that $\lambda(Tor_i^R(X,N)) < \infty $ for all $i \geq 1.$ Since $pd_R(X) = CI\mbox{-}dim_R(M) \leq depth_R(N),$ applying Lemma \ref{prop_depth_acyclicity} to the pair $(X,N),$ we get \begin{align}  \label{j_c}
	Tor_i^R(X,N) = 0 \text{ for } i \geq 1.
	\end{align} 

Using \ref{j_c}, and applying the depth formula on $(X,N)$, the hypothesis \ref{k_a} implies \begin{align} \label{j_d} 
	depth_R(X \otimes_R N)  \geq n. 
\end{align}By \ref{k_b} and \ref{j_c}, we conclude that from the exact sequence $Tor_1^R(X,N) \rar Tor_1^R(G,N) \rar M \otimes_R N,$ that
 \begin{align}  \label{j_x}
	Tor_1^R(G,N) = 0.
\end{align} This gives a short exact sequence $0 \rar M \otimes_R N \rar X \otimes_R N \rar G \otimes_R N \rar 0.$  Now the depth lemma along with \ref{j_d}, and \ref{k_b} gives
\begin{align} \label{j_y}
	depth_R(G \otimes_R N) \geq n. 
\end{align}
	Let $\wtil{G}$ and $\wtil{X}$ be the $Q$-modules associated to the projective hull \eqref{proj_hull} as in Proposition \ref{prop_projective_hull} such that \begin{align} \label{G_approx} 
		0 \rar \wtil{X} \rar \wtil{G} \rar M \rar 0 
	\end{align} is a $\cG(Q)$-approximation of $M$ as a $Q$-module.  Since $CI\mbox{-}dim_Q(M) < \infty,$ the claim in \ref{l_0} follows from Proposition \ref{prop_projective_hull}, \cite[8.7(9)]{Av} and  \ref{section_gorenstein_com}\ref{item_sather}.
	

	Using the change of rings sequence \eqref{seq_change} for the pair $(G,N)$ and \eqref{l_a}, it follows that $\lambda(Tor_i^Q(G,N)) < \infty$ for all $i \geq 2.$ On tensoring the middle horizontal sequence in diagram \eqref{dgm_m}  with $N,$ we conclude that $\lambda(Tor_i^Q(\wtil{G},N)) < \infty \text{ for } i \geq 1,$ which proves \ref{l_1}.

It follows from \eqref{j_c}  and the change of rings sequence \eqref{seq_change}, that $Tor_i^Q(X,N) = 0$ for all $i \geq 2.$ On tensoring the middle vertical sequence in diagram \ref{dgm_m} with $N,$ 
 \begin{align}
		\label{l_6}  Tor_i^Q(\wtil{X},N) = 0 \text{ for } i \geq 1.
	\end{align}

The left vertical sequence in diagram \eqref{dgm_m} and the Tor-independence in \eqref{l_6}, show that the map $Tor_i^Q(\wtil{G},N) \rar Tor_i^Q(M,N)$ is an isomorphism for all $i \geq 2,$ which shows the first half of \ref{l_2}.  Now consider the following part of the change of rings sequence \ref{seq_change} for the pair $(G,N),$ $$Tor_1^R(G,N) \rar Tor_2^Q(G,N) \rar Tor_2^R(G,N) \rar G \otimes_R N \rar Tor_1^Q(G,N) \rar Tor_1^R(G,N).$$ By \ref{j_x}, the first and the last terms vanish. By assumption $n \geq 1,$ therefore the map $Tor_2^R(G,N) \rar G \otimes_R N$ is zero by \ref{l_a}, and  \ref{j_y}. Thus, there are isomorphisms \begin{align} \label{j_z}
	Tor_2^Q(G,N) \cong Tor_2^R(G,N) \text{ and } Tor_1^Q(G,N) \cong G \otimes_R N. 
\end{align}Now, the following isomorphisms follow from diagram \ref{dgm_m} and the vanishing in \ref{j_c} $$Tor_1^Q(\wtil{G},N) \cong Tor_2^Q(G,N) \cong Tor_2^R(G,N) \cong Tor_1^R(M,N),$$ which completes the proof of \ref{l_2}. 

Tensoring the sequence $0 \rar \wtil{G} \rar \wtil{F_0} \rar G \rar 0$ in the diagram \ref{dgm_m} with $N,$ we get $0 \rar Tor_1^Q(G,N) \rar \wtil{G} \otimes_Q N \rar \wtil{F_0} \otimes_Q N \rar G \otimes_Q N \rar 0.$ Now using the depth bound on $G \otimes_R N$ in \ref{j_y}, and the isomorphism \ref{j_z}, we conclude that $depth_Q (\wtil{G} \otimes_Q N) \geq n,$ which proves \ref{l_4}. 

		The last assertion follows from Proposition \ref{prop_projective_hull}.
	\end{proof}

	\begin{proof}[Proof of Theorem \ref{theorem_main_cst}] In view of the depth formula for modules with finite $CI\mbox{-}$dimension (\textit{cf}. section \ref{section_depth_formula}), it is sufficient to show that the following inequality 
		\begin{align} \label{item_x} depth_R(M) + depth_R(N) \geq depth_R(R)  + cx_R(M,N).
		\end{align} implies $Tor_i^R(M,N)=0$ for all $i \geq 1.$ 
		
		Assume first that $cx_R(M,N)=0.$ This means that $Ext_R^i(M,N)=0$ for all $i \gg 0.$ If $R$ is $\mathcal{T}\mathcal{E},$ then by \cite[4.3]{JS} and \eqref{eqn_ET}, or if $R$ is a complete intersection ring, then by \ref{section_gorenstein_com}\ref{item_AvBuch}, we get $Tor_i^R(M,N)=0$ for all $i \gg 0.$  Thus, we are done by Lemma \ref{prop_depth_acyclicity}, which we can apply as $CI\mbox{-}dim_R(M) < \infty;$ see  \ref{section_gorenstein_com}\ref{item_vpd_finite}.
		
		Assume now that $cx_R(M,N) \geq 1.$ We first assume that hypothesis \ref{item_TE} of the Theorem holds, i.e. $R$ is a local complete   $\mathcal{TE}$ ring, $vpd_R(M) < \infty$ and $M$ is locally free on the punctured spectrum of $R.$ By Lemma \ref{lemma_existence_ring} there a local complete  $\mathcal{TE}$ ring $Q$ and a regular element $x\in Q$ such that ${R} =Q/(x)$, $vpd_{Q}({M}) < \infty,$  and $cx_Q({M},{N})=cx_R(M,N)-1.$ 
		
		
		By Lemma \ref{lemma_change_of_ring_tor_vanish_geq_0}, with $n=cx_R(M,N),$  there is a $\cG(Q)$-approximation  $0 \rar \wtil{X} \rar \wtil{G} \rar M \rar 0$ of $M$ on $Q$ such that $CI\mbox{-}dim_Q(\wtil{G}) = 0$ and 
		\begin{enumerate} [label=\normalfont(\Roman*)]
			\item \label{k_1} $\lambda(Tor_i^Q(\wtil{G},N)) < \infty$ for all $i \geq 1.$ 
			\item \label{k_2} $Tor_i^Q(\wtil{G},N) \cong Tor_i^Q(M,N)$ for all $i \geq 2$ and $Tor_1^Q(\wtil{G},N) \cong Tor_1^R(M,N).$   
			\item \label{k_3} $depth_Q (\wtil{G} \otimes_Q N) \geq cx_R(M,N) = cx_Q(\wtil{G},N) + 1.$ 
			\item \label{k_4} $\wtil{G}$ is locally free on the punctured spectrum of $Q.$
		\end{enumerate}
		
		where the equality in \ref{k_3} follows from Lemma \ref{lemma_existence_ring}\ref{item_cx} and the fact that $pd_Q(\wtil{X}) <\infty.$ 
		
		
		Furthermore, it follows from \ref{item_x} that $depth_Q (\wtil{G})+ depth_Q (N) \geq depth_Q (Q)  + cx_Q(\wtil{G},N).$ Thus, the pair $(\wtil{G},N)$ satisfy the inductive hypothesis over $Q$ with $cx_Q(\wtil{G},N) = cx_R(M,N) - 1.$ Therefore, by induction on the complexity, we get \begin{align} \label{vanish} 
			Tor_i^Q(\wtil{G},N) = 0 \text{ for } i \geq 1. 
		\end{align}
		
		It follows from the Tor vanishing in \eqref{vanish}, the first isomorphism in \ref{k_2} that $Tor_i^Q(M,N) = 0$ for all $i \geq 2.$ Now, the change of rings sequence \eqref{seq_change} shows that  \begin{align} \label{periodic}
			Tor_i^R(M,N) \cong Tor_{i+2}^R(M,N) \text{ for } i \geq 1.
		\end{align}  The second isomorphism in \ref{k_2} and \eqref{vanish} gives $Tor_1^R(M,N)=0.$ Now, consider the following part of the exact sequence \eqref{seq_change} $$Tor_2^Q(M,N) \rar Tor_2^R(M,N) \xrightarrow{\Psi}  Tor_0^R(M,N).$$ Here the first term vanishes by \eqref{vanish} and the isomorphism in \ref{k_2}. By hypothesis assumed on $depth_R(M \otimes_R N)$ and $M$ being locally free on the punctured spectrum of $R,$  the map $\Psi$ is $0.$  Thus we conclude that $Tor_2^R(M,N)=0.$ By periodicity of Tor \eqref{periodic}, we conclude that $Tor_i^R(M,N) = 0$ for all $i \geq 1,$ which completes the proof for the case when $R$ is a $\mathcal{TE}$ ring. 
		
		When $R$ is a local complete intersection ring, the only change in the proof is to use Lemma \ref{lemma_existence_ring2} and note that for such rings, the hypothesis $\lambda(Tor_i^R(M,N)) < \infty$ for $i \geq 1$ implies $\lambda(Ext^i_R(M,N)) < \infty$ for $i \geq 1$ by \ref{section_gorenstein_com}\ref{item_AvBuch}. 
	\end{proof}

\begin{remark} If the local ring $R$ satisfies some property $\cP,$ a natural question to ask is whether it is possible to compute  $CI\mbox{-}dim_R(M)$ using a quasi-deformation $R \rar R' \rightarrow Q$ such that $R'$ also satisfies the property $\cP.$ This is true when $\cP$ is the property that $R$ is 
	\begin{enumerate}
		\item a Cohen-Macaulay ring; see \cite[4.1]{CD} and \cite[23.3]{Mat}.
		\item a Gorenstein ring; see \cite[3.1]{Sather} and \cite[23.4]{Mat}.
	\end{enumerate}
	If a similar result holds for the $\mathcal{TE}$ property, one can replace $vpd_R(M)$ with $CI\mbox{-}dim_R(M)$ in part \ref{item_TE} of Theorem \ref{theorem_main_cst}. 
\end{remark}

\begin{remark}
It is not known to us that if the property of being a $\mathcal{TE}$ ring is preserved under localization. If this were true, then in Theorem \ref{theorem_main_cst} one can replace the hypothesis on $M$ being locally free over the punctured spectrum of $R$ with  the hypothesis ``$\lambda(Tor_i^R(M,N)) < \infty$ for all $i \geq 1.$" 
\end{remark}

We show by an example that the hypothesis assumed in Theorem \ref{theorem_main_cst} on depth of tensor product is necessary, and merely having a finiteness condition on length of Tor modules in Theorem \ref{theorem_main_cst} is not enough, even if $M$ and $N$ have sufficiently high depth.
\begin{exmp} \label{eg_3}
	Let $R$ be any (non-regular) Cohen-Macaulay local ring of dimension $d \geq 1.$ Let $M$ be the $d$'th syzygy  $ \Omega^d_Rk$ of the residue field $k.$ Then $M$ is a (nonzero) maximal Cohen-Macaulay $R$-module which is locally free over the punctured spectrum of $R.$ Let $N$ be any (non-free) maximal Cohen-Macaulay  $R$-module.   Observe that $$0 \rar Tor_1^R(\Omega^{d-1}_Rk, N) \rar  M\otimes_R N$$ is a nonzero map and the Tor module has finite length, thus $depth_R(M \otimes_R N) = 0.$ Note that $depth_R(M) = depth_R(N) = depth_R(R)$ and $\lambda(Tor_i^R(M,N)) < \infty$ for all $i \geq 1$ but $Tor_i^R(M,N) \neq 0$ for any $i \geq 1.$ 
\end{exmp}

For any finitely generated module $M$ over a local ring $(R,m,k),$ let $\mu_R(M)$ be the minimal number of generators of $M.$

\begin{exmp}
	Let $k$ be a field and let $R = k[x,y,z]/(x^2-y^2, x^2 - z^2, xy, xz, yz).$ It follows from \cite[3.2.11(b)]{BH} that $R$ is a local Artinian Gorenstein ring. Further, one can check that $m^3 = 0 \neq m^2,$ $\mu_{R}(m) = 3,$ and $\lambda_{R}(R) = 5.$ Thus, $R$ is a local $\mathcal{TE}$ ring which is not a complete intersection; see \cite[3.6]{HJ}. For a family of examples in higher dimension, we refer to \cite[4.2]{KJYT}. 
	The authors are not aware of any example of a $\mathcal{TE}$ ring that is not a complete intersection ring and supports modules with finite virtual projective dimension but infinite projective dimension.
\end{exmp}

As in \cite{CST}, we point out that Theorem \ref{theorem_main_cst} fits in the broader context of the following result.
\begin{theorem}\label{theorem_gen} 
	Let $R$ be a local ring, and let M and N be nonzero finitely generated $R$-modules  Assume that one of the following holds:
	\begin{enumerate}[label=\normalfont(\alph*)]
		\item \label{item_TE_1} $R$ is a local complete  $\mathcal{TE}$ ring, with infinite residue field. Assume that $vpd_R(M) < \infty$ and $M$ is locally free on the punctured spectrum of $R.$
		\item \label{item_CI_1} $R$ is a local complete intersection ring and $\lambda(Tor_i^R(M,N)) < \infty$ for all $i \geq 1.$ 
	\end{enumerate}
	Assume further that $depth_R( M \otimes_R N) \geq cx_R(M,N) + 1.$ Then the following conditions are equivalent:
	\begin{enumerate}[label=\normalfont(\roman*)]
		\item \label{1} $Tor_i^R(M,N) = 0$ for all $i \geq 1.$
		\item \label{2} $depth_R(M) + depth_R(N) \geq depth_R(R) + cx_R(M,N).$
		\item \label{3} $depth_R(M) + depth_R(N) = depth_R(R) + depth_R(M \otimes_R N).$ 
		\item \label{4} $depth_R(M) + depth_R(N) \geq depth_R(R) + depth_R(M \otimes_R N).$ 
		\item \label{5} $depth_R(M) + depth_R(N) \geq depth_R(R)$ and $Tor_i^R(M,N) = 0$ for all $i\gg 0.$
		\item \label{6} $depth_R(M) + depth_R(N) \geq depth_R(R)$ and $Ext^i_R(M,N) = 0$ for all $i\gg 0.$
	\end{enumerate}
\end{theorem}
The part \ref{item_CI_1}  of Theorem \ref{theorem_gen} was proved in \cite[3.1]{CST}. Essentially the same proof applies here and we restate the main steps to keep the article self-contained. 
\begin{proof}[Proof of Theorem \ref{theorem_gen}] The implications \ref{1} $\implies$ \ref{3} and \ref{1} $\implies$ \ref{5} hold more generally by the depth formula \cite{ArYos} or \cite{I}, without any hypothesis on depth of $M \otimes_RN.$ The implications \ref{3} $\implies$ \ref{4} $\implies$ \ref{2} are immediate.  The equivalence of \ref{5} and \ref{6} follows from the definition of a $\mathcal{TE}$ ring and was proved for the complete intersection rings by Avramov and Buchweitz \cite[6.1]{AvBuch}. It holds independently of the depth inequality. It is clear that \ref{6} $\implies$ \ref{2} as the vanishing assumption in \ref{6} implies $cx_R(M,N) = 0.$ 
	
Thus it is sufficient to prove \ref{2} $\implies$ \ref{1}. This follows from the proof of Theorem \ref{theorem_main_cst}; ref inequality \eqref{item_x}.
\end{proof}

	\section{On a question of Celikbas, Sadeghi and Takahashi} 
	In their work \cite{CST}, the authors posed the question of whether the conclusion of Theorem \ref{theorem_gen} holds if one replaces hypothesis $\lambda(Tor_i^R(M,N)) < \infty$ for all $i \geq 1$ with $\lambda(Tor_i^R(M,N)) < \infty$ for all $i \gg 0;$ see \cite[3.10]{CST}. 
	
	In this section, we will first provide a negative answer to this question. We will then establish a weaker bound on depth of tensor product with this relaxed hypothesis on length of Tors in Theorem \ref{theorem_syz}. 
	
	We will construct a pair $(X,C)$ of $R$-modules in Example \ref{eg_4}, such that $\lambda(Tor_i^R(X,C)) < \infty$ for all $i \gg 0$ and $depth_R(X \otimes_R C) \geq cx_R(X,C) + 1.$ However, $X$ and $C$ do not satisfy the implication \ref{2} $\implies$ \ref{1} of Theorem \ref{theorem_gen}, thus answering the above question negatively. 

\begin{exmp} \label{eg_4}
	
	Let $R$ be a two-dimensional local complete intersection ring. Let $X$ be a Cohen-Macaulay $R$-module such that $pd_R(X) = 1.$ For example, we can take $X = R/(x)$ for some regular element $x.$ Observe that $$dim(R/Ann_R(X)) = dim(Supp(X)) = depth_R(X) = 1,$$  therefore $Ann_R(X) \neq 0.$ Set $C = R/Ann_R(X)$ and $N = Ann_R(X).$  Since $Ass(C)\subseteq Ass(X)$ and $depth_R(X) \neq 0,$ we have 
	\begin{align} \label{ineq_d}
		depth_R(C) \geq 1. 
	\end{align} In particular, $depth_R(N) = 2.$ Note also that \begin{align} \label{tor_neq_0} 
		Tor_1^R(X,C) \cong X \otimes_R N \neq 0.
	\end{align}

%

	Since $CI\mbox{-}dim_R(N) = 0,$ and $pd_R(X) = 1,$ $Tor_i^R(X,N) = 0$ for all $i \geq 1.$ The depth formula (section \ref{section_depth_formula}) gives, $$depth_R(X\otimes_R N) = depth_R(X) + depth_R(N) - depth_R(R) = 1.$$ Thus by \eqref{tor_neq_0}, $depth_R(Tor_1^R(X,C)) = 1 \neq 0.$ Observe also that $cx_R(X,C) =0$ as $pd_R(X) < \infty.$ Therefore, the pair $(X,C)$ satisfy the following
	\begin{enumerate}[label=(\alph*)]
		\item $Tor_i^R(X,C) = 0$ for all $i \geq 2.$  
		\item $depth_R(X \otimes_R C) = cx_R(X,C) + 1.$ 
		\item $depth_R(X) + depth_R(C) \geq depth_R(R) + cx_R(X,C).$ 
	\end{enumerate} However, $Tor_1^R(X,C) \neq 0.$ Thus, the pair $(X,C)$ is a counterexample to the question asked in  \cite[3.10]{CST}. 
	
\end{exmp}

    Recall that, $syz_R(M)$ is the largest integer $t \leq depth_R(R)$ such that $M$ is a $t$'th syzygy. We recollect some useful properties of this concept.
    \begin{lemma}\label{lemma_depth_syz} Let $R$ be a Gorenstein local ring. Assume that $M$ is a finitely generated $R$-module such that $G\mbox{-}dim_R(M) < \infty.$ Then
    	\begin{enumerate}[label=\normalfont(\alph*)]
    		\item \label{item_depth_a} $syz_R(M) = depth_R(R)$ if and only if $M$ is totally reflexive. 
    		\item \label{item_depth_b} $depth_R(M) \geq syz_R(M).$
    		\item \label{item_depth_c} If $0 \rar M \rar X \rar G \rar 0$ is a projective hull of $M$ (cf. (\ref{section_gorenstein}\ref{item_proj_hull})), then $syz_R(X) \geq syz_R(M).$ 
    	\end{enumerate}
    \end{lemma}
      \begin{proof} Set $d = depth_R(R)$ and $s = syz_R(M).$ If $s= d,$ there is a resolution $0 \rar M \rar F_{d-1} \rar \cdots \rar F_0 \rar M' \rar 0.$ A quick check shows that $d \geq depth_R(M) \geq min(depth_R(M') + d, depth_R(R)) \geq d.$ Converse is immediate. This proves \ref{item_depth_a} 
    	
    	By \ref{section_gorenstein}\ref{item_AB}, $depth_R(M) \leq depth_R(R).$ If equality holds, then $M$ is totally reflexive and hence $s = d = depth_R(M),$ by \ref{item_depth_a}. So we may assume strict inequality. By depth lemma, $depth_R(\Omega_RM) = depth_R(M) + 1.$ By decreasing induction hypothesis on $depth$, we can assume $depth_R(\Omega_RM) \geq syz_R(\Omega_RM).$ Combining the two, it is sufficient to show that $syz_R(\Omega_RM) \geq s + 1,$ which is immediate from definition and \ref{item_depth_a}.
    	
    	We now prove \ref{item_depth_c}. If $M$ is totally reflexive, then $X$ is free; hence $s = syz_R(X) = d,$ and we are done. So assume that $$d > depth_R(M) = depth_R(X) \geq \max(s,syz_R(X)).$$ Since $M$ is an $s$'th syzygy, it is $s$-torsionless; see \cite[42]{Masek}. Now observe that by \cite[8(a)]{Masek},  $X$ is also $s$-torsionless, hence, again by \cite[42]{Masek}, $X$ is an $s$'th syzygy. This shows that $syz_R(X) \geq syz_R(M).$ 
    \end{proof}

	The proof of Theorem \ref{theorem_syz} requires a technical Lemma \ref{lemma_3_7_new}, which is similar to \cite[3.7]{CST}. A word about the trade off between the two lemmas - the hypothesis \ref{item_1} of Lemma  \ref{lemma_3_7_new} is as desired in  \cite[3.10]{CST}, the inequality in Lemma \ref{lemma_depth_syz} show that this comes at the expense of a stronger assumption in \ref{item_2}. 
	\begin{lemma} \label{lemma_3_7_new} Let $R$ be a Gorenstein local ring, and let $M$ and $N$ be nonzero finitely generated $R$-modules such that $CI\mbox{-}dim_R(M)< \infty.$ Let $c$ be any non-negative integer for which the following conditions hold:
		\begin{enumerate}[label=\normalfont(\alph*)] 
			\item \label{item_1}  $\lambda(Tor_i^R(M,N)) < \infty$ for all $i \gg 0.$
			\item \label{item_2} $syz_R(M) + syz_R(N) \geq depth_R(R) + c.$
			\item \label{item_3} $depth_R(M \otimes_R N) \geq c+1.$
		\end{enumerate}
		Then there exists a finitely generated $R$-module $L$ satisfying the following
		\begin{enumerate}[label=\normalfont(\roman*)]
			\item \label{item_a_0} $CI\mbox{-}dim_R(L) = 0.$ 
			\item \label{item_a_1} $Tor_i^R(L,N) = 0$ for $i = 1,\ldots, c+1.$
			\item \label{item_a_2} $Tor_{i+c+1}^R(L,N) \cong Tor_i^R(M,N)$ for all $i \geq 1.$
			\item \label{item_a_3} $cx(M,N) = cx_R(L,N).$ 
		\end{enumerate}
	\end{lemma}

	\begin{proof}[Proof of Lemma \ref{lemma_3_7_new}] Consider the projective hull of $M$  (\ref{section_gorenstein}\ref{item_proj_hull}) \begin{align}\label{seq_proj_hull} 
			0 \rar M \rar X \rar G \rar 0.
		\end{align} Recall that by definition of a projective hull, $pd_R X = CI\mbox{-}dim_R(M)<\infty$ and $G$ is totally reflexive. By \ref{section_gorenstein_com}\ref{item_sather}, we can assume $CI\mbox{-}dim_R(G) = 0.$ Set $m = syz_R(M)$ and $n = syz_R(N).$ By Lemma \ref{lemma_depth_syz}\eqref{item_depth_c}, $X$ is an $m$'th syzygy of an $R$-module $X_1$ and $N$ is an $n$'th syzygy of an $R$-module $N_1.$ In particular, $Tor_i^R(X_1, N_1) = 0$ for all $i \gg 0.$ By Auslander-Buchsbaum formula, $pd_R(X_1) \leq depth_R(R),$ thus $Tor_i^R(X_1, N_1) = 0$ for all $i > depth_R(R).$  Thus, the hypothesis  \ref{item_2} implies, \begin{align} \label{tor_x_n_0} 
			Tor_i^R(X,N) = 0 \ \text{for all }  i\geq 1. 
		\end{align} Applying Auslander's depth formula \ref{section_depth_formula} to the Tor-independent pair $(X,N),$ and using Lemma \ref{lemma_depth_syz}, we get \begin{align} \label{depth_x_otimes_n} 
			 \nonumber depth_R(X \otimes_R N) &= depth_R(X) + depth_R(N) - depth_R(R) \\  & \geq m + n - depth_R(R) \geq c.  
		\end{align}It follows from \eqref{seq_proj_hull} and \eqref{tor_x_n_0} that $Tor_i^R(M,N) \cong Tor_{i+1}^R(G,N)$ for all $i \geq 1.$ Thus, by hypothesis \ref{item_1} and \ref{section_gorenstein_com}\ref{item_ci_dim_0} \begin{align} \label{tor_length_g_n} 
			\lambda(Tor_i^R(G,N)) < \infty \text{ for all } i \geq 1.
		\end{align} 
		
		Tensoring \eqref{seq_proj_hull} with $N$, and using \eqref{tor_x_n_0} gives $0 \rar Tor_1^R(G,N) \rar M \otimes_R N.$ In view of \ref{item_3} and \eqref{tor_length_g_n}, we conclude that $Tor_1^R(G,N) = 0.$ Thus there exists a short exact sequence $0 \rar M \otimes_R N \rar X \otimes_R N \rar G \otimes_R N \rar 0.$ By \ref{item_3} and \eqref{depth_x_otimes_n}, $depth_R(G \otimes_R N) \geq c.$ It is clear that $Ext^i(M,N) \cong Ext^{i+1}(G,N)$ for all $i \gg 0,$ thus $cx_R(M,N) = cx_R(G,N).$ This shows that $G$ satisfies the following properties
		\begin{enumerate}
			\item $CI\mbox{-}dim_R(G) = 0.$
			\item $Tor_1^R(G,N) = 0.$
			\item $Tor_{i+1}^R(G,N) \cong Tor_i^R(M,N)$ for all $i \geq 1.$
			\item $cx_R(M,N) = cx_R(G,N).$ 
			\item $depth_R(G \otimes_R N) \geq c.$
		\end{enumerate}
		If $c= 0,$ we take $L = G.$ If $c \geq 1,$ we repeat this process by replacing $M$ by $G,$ taking the universal pushforward of $G,$ and replacing $c$ by $c-1.$ 
	\end{proof}

	\begin{proof}[Proof of theorem \ref{theorem_syz}] Assuming the contrary, we have an inequality	\begin{align}  \label{item_b} syz_R(M) + syz_R(N) > depth_R(R) + depth_R(M \otimes_R N).
		\end{align} Applying Lemma \ref{lemma_3_7_new}, with $c = cx_R(M,N),$ gives an $R$-module $L$ with
	\begin{enumerate}
		\item $CI\mbox{-}dim_R(L) = 0,$ 
		\item $Tor_i^R(L,N) = 0 \text{ for } i = 1,\ldots, cx_R(L,N) +1,$
		\item $Tor_{i+c+1}^R(L,N) \cong Tor_i^R(M,N) \text{ for all } i \geq 1.$
		\item $cx_R(M,N) = cx_R(L,N),$ 
	\end{enumerate} By Theorem \cite[3.4]{CST} applied to the pair $(L,N),$ it follows that $Tor_i^R(M,N) = 0$ for all $i \geq 1.$ Now  the depth formula and the inequality \ref{lemma_depth_syz}\ref{item_depth_b} gives a contradiction.

	\end{proof}

\begin{remark}
	The hypothesis on the lengths of $Tor$s as used in Theorems \ref{theorem_main_cst} and \ref{theorem_syz} is satisfied, for instance, if $R$ has an isolated singularity or if either of the modules has locally finite projective dimension over the punctured spectrum of $R.$
\end{remark}

	\subsection*{Acknowledgement} 
	We are grateful to Olgur Celikbas for the valuable feedback that helped in improving this article. We thank Dipankar Ghosh for many suggestions and, in particular, for pointing out an error in an earlier version of this manuscript. We thank Kaito Kimura, Yuya Otake, Ryo Takahashi and Shashi Ranjan Sinha for many useful comments and suggestions. We thank the anonymous referee for a careful reading and useful suggestions that have improved the exposition.  
	
	The first author is partially supported by a NET Senior Research Fellowship from UGC, MHRD, Govt. of India.

\end{document}